\documentclass[11pt]{amsart}

\usepackage{amsfonts}
\usepackage{amssymb}
\usepackage{amsthm}
\usepackage{amsmath}
\usepackage{hyperref}
\usepackage{enumitem}
\usepackage{tikz}
\usepackage{pgf}
\usepackage{graphicx}
\usepackage[top=3.5cm, bottom=3.5cm, left=2.5cm, right=2.5cm]{geometry}
\usepackage[capitalise]{cleveref}

\newtheorem{prop}{Proposition}[section]
\newtheorem{theorem}[prop]{Theorem}
\newtheorem{lemma}[prop]{Lemma}
\newtheorem{corollary}[prop]{Corollary}

\theoremstyle{definition}
\newtheorem*{definition*}{Definition}

\theoremstyle{remark}

\newtheorem*{remark*}{Remark}
\newtheorem{remark}[prop]{Remark}

\theoremstyle{theorem}

\newcommand{\ph}{\varphi}
\newcommand{\R}{\mathbb{R}}
\newcommand{\C}{\mathbb{C}}

\newcommand{\N}{\mathbb{N}}
\newcommand{\Z}{\mathbb{Z}}

\newcommand{\F}{\mathbb{F}}

\newcommand{\Sym}{\text{\textup{Sym}}}

\newcommand{\ord}{\text{\textup{ord}}}
\newcommand{\nil}{\text{\textup{nil}}}

\newcommand{\Upp}{\text{\textup{Upp}}}

\numberwithin{equation}{section}

\title{Polylogarithmic bounds in the nilpotent Freiman theorem}
\date{}
\author{Matthew C. H. Tointon}
\address{Pembroke College, Cambridge, CB2 1RF, United Kingdom}
\email{mcht2@cam.ac.uk}
\thanks{The author is the Stokes Research Fellow at Pembroke College, Cambridge}

\begin{document}
\maketitle

\begin{abstract}We show that if $A$ is a finite $K$-approximate subgroup of an $s$-step nilpotent group then there is a finite normal subgroup $H\subset A^{K^{O_s(1)}}$ modulo which $A^{O_s(\log^{O_s(1)}K)}$ contains a nilprogression of rank at most $O_s(\log^{O_s(1)}K)$ and size at least $\exp(-O_s(\log^{O_s(1)}K))|A|$. This partially generalises the close-to-optimal bounds obtained in the abelian case by Sanders, and improves the bounds and simplifies the exposition of an earlier result of the author. Combined with results of Breuillard--Green, Breuillard--Green--Tao, Gill--Helfgott--Pyber--Szab\'o, and the author, this leads to improved rank bounds in Freiman-type theorems in residually nilpotent groups and certain linear groups of bounded degree.
\end{abstract}

\tableofcontents

\section{Introduction}
This paper concerns \emph{sets of small doubling} and \emph{approximate groups} in non-abelian groups. This topic has been extensively covered in the recent mathematical literature; the reader may consult the author's forthcoming book \cite{book} or the surveys \cite{bgt.survey,app.grps,ben.icm,helf.survey,sand.survey,raconte-moi} for detailed background to the topic and examples of some of its many applications.

Given sets $A$ and $B$ in a group $G$ we define the \emph{product set} $AB$ by $AB=\{ab:a\in A,b\in B\}$, and define $A^n$ recursively for $n\in\N$ by setting $A^1=A$ and $A^{n+1}=A^nA$. We also write $A^{-1}=\{a^{-1}:a\in A\}$ and $A^{-n}=(A^{-1})^n$. If $G$ is abelian we often use additive notation instead, for example writing $A+B$ or $nA$ in place of $AB$ or $A^n$, respectively.

By the \emph{doubing} of a finite set $A$ we mean the ratio $|A^2|/|A|$, and when we say that a set has `small' or `bounded' doubling we mean that there is some constant $K\ge1$ such that $|A^2|\le K|A|$. Of course, this always holds for $K=|A|$, so $K$ should be thought of as being substantially smaller than $|A|$ in order for this to be meaningful.

One of the central aims in the theory of sets of small doubling is to describe the algebraic structure of such sets. The first result in this direction was Freiman's theorem \cite{freiman}, which describes sets of small doubling in terms of objects called \emph{progressions}. Given elements $x_1,\ldots,x_r$ in an abelian group $G$ and reals $L_1,\ldots,L_r\ge0$, the  \emph{progression} $P(x;L)$ is defined via $P(x;L)=\{\ell_1x_1+\cdots+\ell_rx_r:|\ell_i|\le L_i\}$. Freiman's theorem states that if a subset $A\subset\Z$ satisfies $|A+A|\le K|A|$ then there exists a progression $P$ of rank at most $r(K)$ and size at most $h(K)|A|$ such that $A\subset P$. This was subsequently generalised to an arbitrary abelian group by Green and Ruzsa \cite{green-ruzsa}, where one must replace the progression with a \emph{coset progression}, which simply means a set of the form $H+P$, with $H$ a finite subgroup and $P$ a progression.

The best bounds currently available in this theorem are due to Sanders (although Schoen \cite{schoen} had previously obtained similar bounds in the special case of subsets of integers). Sanders's main result is the following variant of Freiman's theorem; in it and elsewhere we write $\log^mK$ to mean $(\log K)^m$.

\begin{theorem}[Sanders {\cite[Theorem 1.1]{sanders}}]\label{thm:sanders}
Let $A$ be a finite subset of an abelian group such that $|A+A|\le K|A|$. Then there exists a coset progression $H+P$ of rank at most $O(\log^{O(1)}2K)$ such that $H+P\subset 2A-2A$ and $|H+P|\ge\exp(-O(\log^{O(1)}2K))|A|$.
\end{theorem}

Combining \cref{thm:sanders} with the so-called \emph{covering argument} of Chang \cite{chang}---which we present in \cref{lem:chang}, below---one obtains the following bounds in Freiman's theorem.
\begin{corollary}[Sanders]\label{cor:sanders.ext}
Let $A$ be a finite subset of an abelian group such that $|A+A|\le K|A|$. Then there exists a coset progression $H+P$ of rank at most $O(K\log^{O(1)}2K)$ satisfying $|H+P|\le\exp(O(K\log^{O(1)}2K))|A|$ such that $A\subset H+P$.
\end{corollary}

These bounds are close to best possible, as can be seen by considering, for example, an appropriate union of $K$ translates of a finite subgroup or a rank-$1$ progression.

It is worth remarking that using a simpler covering argument due to Ruzsa \cite{ruzsa}, on which Chang's argument is based, one can also deduce the following variant of Theorem \ref{thm:sanders}; we give details in Section \ref{sec:prelim}.

\begin{corollary}[Sanders]\label{cor:sanders}
Let $A$ be a finite subset of an abelian group such that $|A+A|\le K|A|$. Then there exists a coset progression $H+P\subset 4A-4A$ of rank at most $O(\log^{O(1)}2K)$ satisfying $|H+P|\le K^8|A|$, and a set $X\subset A$ of size at most $\exp(O(\log^{O(1)}2K))$ such that $A\subset X+H+P$.
\end{corollary}

In this paper we are concerned with generalisations of these results to non-abelian groups, and specifically to nilpotent groups. The basic properties of nilpotent groups that we use can be found in \cite[Chapter 10]{hall} or \cite[\S5.2]{book}.

In the non-abelian setting it is usual for technical reasons to replace the small-doubling assumption $|A+A|\le K|A|$ with a slightly stronger assumption. This is usually either a `small-tripling' assumption $|A^3|\le K|A|$, or the qualitativey even stronger assumtion that $A$ is a \emph{$K$-approximate group}. 
\begin{definition*}[approximate group]
Given $K\ge1$, a subset $A$ of a group $G$ is said to be a \emph{$K$-approximate subgroup of $G$}, or simply a \emph{$K$-approximate group}, if $A^{-1}=A$ and $1\in A$, and if there exists $X\subset G$ with $|X|\le K$ such that $A^2\subset XA$.
\end{definition*}
The reasons for making these stronger assumptions are explained at length in \cite{tao.product.set,book}, but let us highlight the fact that a set $A$ satisfying $|A^2|\le K|A|$ is contained in the union of a few translates of a relatively small $O(K^{O(1)})$-approximate group \cite[Theorem 4.6]{tao.product.set}, so there is no great loss of generality in doing so. Note, conversely, that if $A$ is a finite $K$-approximate group then $|A^m|\le K^{m-1}|A|$ for every $m\in\N$, a fact we will use on a number of occasions without further mention.

There are a number of ways to formulate the appropriate generalisation of a coset progression to non-abelian groups. The easiest to define is probably a \emph{coset nilprogression}. Given elements $x_1,\ldots,x_r$ in a group $G$ and  $L_1,\ldots,L_r\ge0$, the \emph{nonabelian progression} $P(x;L)$ is defined to consist of all those elements of $G$ that can be expressed as words in the $x_i$ and their inverses in which each $x_i$ and its inverse appear at most $L_i$ times between them. We define $r$ to be the \emph{rank} of $P(x;L)$. If the $x_i$ generate an $s$-step nilpotent group then $P(x;L)$ is said to be a \emph{nilprogression} of step $s$, and in this instance we write $P_\nil(x;L)$ instead of $P(x;L)$. A set $P$ is said to be a \emph{coset nilprogression} of rank $r$ and step $s$ if there exists a finite subgroup $H\subset P$, normalised by $P$, such that the image of $P$ in $\langle P\rangle/H$ is a nilprogression of rank $r$ and step $s$.

Another useful formulation is a closely related object called a \emph{nilpotent progression}. Again, a nilpotent progression $\overline P(x;L)$ is defined using elements $x_1,\ldots,x_r$ in a nilpotent group $G$ and reals $L_1,\ldots,L_r\ge0$, but its definition is a little more involved than that of a nilprogression, so we refer the reader to any of \cite{bg,nilp.frei,book}.

Nilpotent progressions have tripling bounded in terms of their rank and step, as do nilprogressions if the reals $L_1$ are large enough \cite[Corollary 3.16]{bt}. For technical reasons, it is also convenient to define a third type of progression in a non-abelian group, although in general this one will not have bounded doubling. Given $x_i$ and $L_i$ as above, the \emph{ordered progression} $P_\ord(x;L)$ is defined to be $P_\ord(x;L)=\{x_1^{\ell_1}\cdots x_r^{\ell_r}:|\ell_i|\le L_i\}$.

The following result shows that it does not matter too much which of the above versions of progression we use.

\begin{prop}[{\cite[Proposition C.1]{nilp.frei}}]\label{prop:nilprog.equiv}
Let $G$ be an $s$-step nilpotent group, let $x_1,\ldots,x_r\in G$, and let $L_1,\ldots,L_r\in\N$. Then $P_\ord(x;L)\subset P_\nil(x;L)\subset\overline P(x;L)\subset P_\ord(x;L)^{(96s)^{s^2}r^s}$.
\end{prop}
\begin{proof}[Remarks on the proof]The bounds we state here are written more explicitly than in \cite[Proposition C.1]{nilp.frei}, but bounds of the type we claim here can easily be read out of the argument there. \cref{prop:nilprog.equiv} is also proved exactly as stated above in \cite[Proposition 5.6.4]{book}.
\end{proof}

The author has previously extended \cref{cor:sanders.ext} to nilpotent groups, proving the following result.

\begin{theorem}[{\cite[Theorem 1.5]{nilp.frei}}]\label{thm:old}
Let $s\in\N$ and $K\ge1$. Let $G$ be an $s$-step nilpotent group $s$, and suppose that $A\subset G$ is a finite $K$-approximate group. Then there exist a subgroup $H$ of $G$ normalised by $A$ and a nilprogression $P_\nil(x;L)$ of rank at most $K^{O_s(1)}$ such that
\[
A\subset HP_\nil(x;L)\subset H\overline P(x;L)\subset A^{K^{O_s(1)}}.
\]
\end{theorem}
\begin{remark*}
In particular, $|H\overline P(x;L)|\le\exp(K^{O_s(1)})|A|$.
\end{remark*}

The aim of the present paper is to show that, like in the abelian case, if we ask for $HP$ to be dense in $A$, rather than the other way around, we can replace most of the polynomial bounds of \cref{thm:old} with polylogarithmic bounds, as follows.

\begin{theorem}\label{thm:new.gen}
Let $s\in\N$ and $K\ge1$. Let $G$ be an $s$-step nilpotent group, and suppose that $A\subset G$ is a finite $K$-approximate group. Then there exist a subgroup $H\subset A^{K^{e^{O(s)}}}$ normalised by $A$ and an ordered progression $P_\ord(x;L)\subset A^{e^{O(s^2)}\log^{O(s)}2K}$ of rank at most $e^{O(s^2)}\log^{O(s)}2K$ such that
\[
P_\ord(x;L)\subset P_\nil(x;L)\subset\overline P(x;L)\subset A^{e^{O(s^3)}\log^{O(s^2)}2K}
\]
and
\[
|HP_\ord(x;L)|\ge\exp\left(-e^{O(s^2)}\log^{O(s)}2K\right)|AH|.
\]
\end{theorem}

The proof of \cref{thm:old} proceeds by an induction on the step $s$, in which \cref{thm:sanders} features both in the base case $s=1$ and in the proof of the inductive step. The original proof used an earlier version of \cref{thm:sanders}, due to Green and Ruzsa, in which the bounds are polynomial rather than polylogarithmic. Let us emphasise, though, that losses elsewhere in the argument overwhelmed the bounds of \cref{thm:sanders} to the extent that it made no difference to the shape of the final bounds to use the Green--Ruzsa result instead. In particular, proving \cref{thm:new.gen} is not merely a case of substituting \cref{thm:sanders} for the Green--Ruzsa result in the original proof: we also need to make the rest of the argument more efficient.

The one bound that is still polynomial in \cref{thm:new.gen} is the bound $H\subset A^{K^{e^{O(s)}}}$; it appears that a new idea would be required to improve this any further (see \cref{rem:poly.bound}, below, for further details). Note, though, that in the case where the ambient group has no torsion the subgroup $H$ is automatically trivial, leaving only the polylogarithmic bounds, as follows.

\begin{theorem}\label{thm:new.tf}
Let $s\in\N$ and $K\ge1$. Let $G$ be a torsion-free $s$-step nilpotent group, and suppose that $A\subset G$ is a finite $K$-approximate group. Then there exist an ordered progression $P_\ord(x;L)\subset A^{e^{O(s^2)}\log^{O(s)}2K}$ of rank at most $e^{O(s^2)}\log^{O(s)}2K$ such that
\[
P_\ord(x;L)\subset P_\nil(x;L)\subset\overline P(x;L)\subset A^{e^{O(s^3)}\log^{O(s^2)}2K}
\]
and
\[
|P_\ord(x;L)|\ge\exp\left(-e^{O(s^2)}\log^{O(s)}2K\right)|A|.
\]
\end{theorem}

As in the abelian case, Ruzsa's covering argument combines with \cref{thm:new.gen} to give the following variant.

\begin{corollary}\label{cor:ruzsa}
Let $s\in\N$ and $K\ge1$. Let $G$ be an $s$-step nilpotent group, and suppose that $A\subset G$ is a finite $K$-approximate group. Then there exist a subgroup $H\subset A^{K^{e^{O(s)}}}$ normalised by $A$, a nilprogression $P_\nil(x;L)$ of rank at most $e^{O(s^2)}\log^{O(s)}2K$ such that
\[
P_\nil(x;L)\subset\overline P(x;L)\subset A^{e^{O(s^3)}\log^{O(s^2)}2K},
\]
and a subset $X\subset G$ of size at most $\exp(e^{O(s^2)}\log^{O(s)}2K)$ such that $A\subset XHP_\nil(x;L)$.
\end{corollary}
\begin{remark*}
In particular, $|H\overline P(x;L)|\le\exp(K^{e^{O(s)}})|A|$. In the torsion-free setting the subgroup $H$ is again trivial, and in that case we may conclude instead that  $|\overline P(x;L)|\le\exp(e^{O(s^3)}\log^{O(s^2)}2K)|A|$.
\end{remark*}

Chang's covering argument also allows us to recover \cref{thm:old} with much more precise bounds, as follows.

\begin{corollary}\label{cor:chang.ag}
Let $s\in\N$ and $K\ge1$. Let $G$ be an $s$-step nilpotent group, and suppose that $A\subset G$ is a finite $K$-approximate group. Then there exist a subgroup $H\subset A^{K^{e^{O(s)}}}$ normalised by $A$ and a nilprogression $P_\nil(x;L)$ of rank at most $e^{O(s^2)}K\log^{O(s)}2K$ such that
\[
A\subset HP_\nil(x;L)\subset H\overline P(x;L)\subset HA^{e^{O(s^3)}K^{s+1}\log^{O(s^2)}2K}.
\]
\end{corollary}
\begin{remark*}
In particular, $|H\overline P(x;L)|\le\exp(K^{e^{O(s)}})|A|$, or $|P(x;L)|\le\exp(e^{O(s^3)}K^{s+1}\log^{O(s^2)}2K)|A|$ in the torsion-free setting.
\end{remark*}

We deduce these corollaries in \cref{sec:covering}.

\bigskip\noindent\textsc{Applications to other groups.}
A theorem of Breuillard, Green and Tao \cite{bgt} states, in one form, that an arbitrary finite $K$-approximate group $A$ is contained in a union of at most $O_K(1)$ translates of a coset nilprogression of rank and step $O(K^2\log K)$ and size at most $K^{11}|A|$. This result is powerful enough to have some quite general applications, such as those contained in \cite[\S11]{bgt} and \cite{bt,tao.growth,tt}, but its usefulness is slightly lessened by the fact that it does not give an explicit bound on the number of translates needed to contain $A$. Partly for this reason, various papers by several different authors have given explicit versions of this theorem for certain specific classes of groups.

The approach taken in these results is generally first to reduce to the nilpotent case, and then to apply \cref{thm:old} (or an earlier result of Breuillard and Green \cite{bg} valid only in the torsion-free setting) to obtain the nilprogression. Unsurprisingly, using \cref{thm:new.gen} or one of its corollaries in place of \cref{thm:old} in these arguments leads to better bounds in a number of cases. In \cref{sec:non-nilp} we present such better bounds for linear groups over $\F_p$ or fields of characteristic zero, and in residually nilpotent groups.

\bigskip\noindent\textsc{Acknowledgement.}
I was prompted to revisit the bounds in \cref{thm:old} by a question from Harald Helfgott.

\section{Standard tools}\label{sec:prelim}
In this section we record various standard results relating to sets of small doubling and approximate groups. This material is likely to be familiar to experts in the subject, who may therefore decide to skip straight to \cref{sec:details}.

\begin{lemma}[{\cite[Proposition 2.6.5]{book}}]\label{lem:slicing}
Let $K,L\ge1$ and let $G$ be a group. Let $A\subset G$ be a $K$-approximate group and $B\subset G$ an $L$-approximate group. Then for every $m,n\ge2$ the set $A^m\cap B^n$ is covered by at most $K^{m-1}L^{n-1}$ left translates of $A^2\cap B^2$. In particular, $A^m\cap B^n$ is a $K^{2m-1}L^{2n-1}$-approximate group.
\end{lemma}

\begin{lemma}\label{lem:fibre.pigeonhole}
Let $k\in\N$. Let $G$ be a group with a subgroup $H$, let $A\subset G$, and suppose that $A$ is contained in a union of $k$ left cosets of $H$. Then $A$ is contained in a union of $k$ left translates of $A^{-1}A\cap H$.
\end{lemma}
\begin{proof}Let $x_1,\ldots,x_m\in A$ be representatives of the left cosets of $H$ containing at least one element of $A$, noting that $m\le k$ by hypothesis. If $a$ is an arbitrary element of $A\cap x_iH$ then there exists $h\in H$ such that $a=x_ih$. It follows that $h=x_i^{-1}a\in A^{-1}A$, and hence $h\in A^{-1}A\cap H$ and $a\in x_i(A^{-1}A\cap H)$.
\end{proof}

\begin{theorem}[Pl\"unnecke's inqequalities \cite{petridis}]\label{thm:plun}
Let $G$ be an abelian group, and let $A$ be a finite subset of $G$. Suppose that $|A+A|\le K|A|$. Then $|mA-nA|\le K^{m+n}|A|$ for all non-negative integers $m,n$.
\end{theorem}

\begin{lemma}[Ruzsa's covering lemma {\cite[Lemma 5.1]{bgt}}]\label{lem:covering}
Let $A$ and $B$ be finite subsets of some group and suppose that $|AB|/|B|\le K$. Then there exists a subset $X\subset A$ with $|X|\le K$ such that $A\subset XBB^{-1}$.
\end{lemma}

\begin{proof}[Proof of \cref{cor:sanders}]
Let $H$ and $P$ be as given by \cref{thm:sanders}. Then we have
\begin{align*}
\frac{|A+H+P|}{|H+P|}&\le\exp(O(\log^{O(1)}2K))\frac{|A+H+P|}{|A|}\\
   &\le K^5\exp(O(\log^{O(1)}2K))&\text{(by \cref{thm:plun})}\\
   &\le\exp(O(\log^{O(1)}2K)),
\end{align*}
and so \cref{lem:covering} gives a set $X\subset A$ of size at most $\exp(O(\log^{O(1)}2K))$ such that $A\subset X+H+2P$. Now $2P$ is also a progression of the same rank as $P$. Moreover, since $H+P\subset 2A-2A$, we have $H+2P\subset 4A-4A$, and hence  $|H+2P|\le K^8|A|$ by \cref{thm:plun}. This completes the proof.
\end{proof}

\begin{lemma}[Chang's covering lemma {\cite[Proposition 2.4]{nilp.frei}}]\label{lem:chang}
Let $K,C\ge1$ and $m\in\N$. Let $G$ be a group, and suppose that $A\subset G$ is a finite $K$-approximate group. Suppose that $B\subset A^m$ is a set with $|B|\ge|A|/C$. Then there exist $t\ll\log C+m\log K$ and sets $S_1,\ldots,S_t\subset A$ satisfying $|S_i|\le2K$ such that $A\subset S_{t-1}^{-1}\cdots S_1^{-1}B^{-1}BS_1\cdots S_t$.
\end{lemma}

\begin{definition*}[Freiman homomorphism]
Let $k\in\N$, let $G$ be a group, and let $A$ be a subset of a group. Then a map $\varphi:A\to G$ is a \emph{Freiman $k$-homomorphism}, or simply a \emph{$k$-homomorphism}, if whenever $x_1,\ldots,x_k,y_1,\ldots,y_k\in A$ satisfy
\[
x_1\cdots x_k=y_1\cdots y_k
\]
we have
\[
\ph(x_1)\cdots\ph(x_k)=\ph(y_1)\cdots\ph(y_k).
\]
If $1\in A$ and $\ph(1)=1$ then we say that $\ph$ is \emph{centred}.
\end{definition*}

\begin{lemma}\label{lem:fr.hom.ap.grp}
Let $K\ge1$. Let $A$ be a $K$-approximate group, let $G$ be a group. Suppose that $\ph:A\to G$ is a centred Freiman $3$-homomorphism. Then $\ph(A)$ is a $K$-approximate group.
\end{lemma}
\begin{proof}
The set $\ph(A)$ contains the identity by definition of a centred Freiman homomorphism. Moreover, for every $a\in A$ we have $\ph(a^{-1})\ph(a)\ph(1)=\ph(1)^3$, and hence
\begin{equation}\label{eq:centred}
\ph(a^{-1})=\ph(a)^{-1},
\end{equation}
so $\ph(A)$ is symmetric. Finally, by definition there is a set $X$ of size at most $K$ such that $A^2\subset XA$. We may assume that $X$ is minimal satisfying this property, and hence that $X\subset A^3$. For each $x\in X$ there therefore exist elements $\alpha_1(x),\alpha_2(x),\alpha_3(x)\in A$ such that $x=\alpha_1(x)\alpha_2(x)\alpha_3(x)$. Set $Y=\{\ph(\alpha_1(x))\ph(\alpha_2(x))\ph(\alpha_3(x)):x\in X\}$, noting that $|Y|\le K$. We claim that $\ph(A)^2\subset Y\ph(A)$, which will complete the proof. To prove this claim, fix $a_1,a_2\in A$, and let $x\in X$ and $a_3\in A$ be such that $a_1a_2=xa_3$. It follows from \eqref{eq:centred} that $\ph(a_1)\ph(a_2)\ph(a_3)^{-1}=\ph(\alpha_1(x))\ph(\alpha_2(x))\ph(\alpha_3(x))$, and so $a_1a_2\in Y\ph(A)$ as claimed.
\end{proof}

\section{Proof of the main result}\label{sec:details}
Before we prove \cref{thm:new.gen}, let us remark that at various points we make the seemingly unnecessary assumption that $K\ge2$. The reason for this is purely notational: it allows us to replace bounds such as $O(\log^{O(1)}2K)$ or $O(K^{O(1)})$ with the slightly more succinct $\log^{O(1)}2K$ or $K^{O(1)}$, respectively. Note that if $K<2$ then a $K$-approximate group is an actual subgroup, in which regime all of our main results become trivial, so we lose nothing in making this assumption.

We start the proof of \cref{thm:new.gen} with the following result, a version of which with worse bounds was also central to the original proof of \cref{thm:old}.

\begin{prop}\label{prop:key.general}
Let $m>0$ and $s\ge\tilde s\ge2$ be integers, and let $K,\tilde K\ge2$. Let $G$ be an $s$-step nilpotent group generated by a finite $K$-approximate group, and let $\tilde A\subset A^m$ be a $\tilde K$-approximate group that generates an $\tilde s$-step nilpotent subgroup $\tilde G$ of $G$. Then there exist a normal subgroup $N\lhd G$ with $N\subset A^{K^{e^{O(s)}m}}$, an integer $r\le\log^{O(1)}2\tilde K$, and $\tilde K^{O(1)}$-approximate groups $A_0,\ldots,A_r\subset\tilde A^{O(1)}$ such that
\[
|A_0\cdots A_r|\ge\frac{|\tilde A|}{\exp(\log^{O(1)}2\tilde K)},
\]
and such that, writing $\rho:G\to G/N$ for the quotient homomorphism, each group $\langle\rho(A_i)\rangle$ has step less than $\tilde s$.
\end{prop}

The main ingredients in the proof of \cref{prop:key.general} are the next three results.

\begin{prop}\label{prop:pre-chang.tor.free}
Let $s\ge2$ and $K\ge2$. Let $G$ be an $s$-step nilpotent group, and write $\pi:G\to G/[G,G]$ for the quotient homomorphism. Suppose that $A\subset G$ is a finite $K$-approximate group. Then there exist an integer $r\le\log^{O(1)}2K$, elements $x_1,\ldots,x_r\in\pi(A^4)$, and a subgroup $H\subset\pi(A^4)$ such that 
\[
\left|\left(A^{18}\cap\pi^{-1}(H)\right)\prod_{i=1}^r\left(A^{24}\cap\pi^{-1}(\langle x_i\rangle)\right)\right|\ge\frac{|A|}{\exp(\log^{O(1)}2K)}.
\]
\end{prop}
We prove \cref{prop:pre-chang.tor.free} shortly.

\begin{lemma}\label{lem:x[G,G]}
Let $s\ge2$. Let $G$ be an $s$-step nilpotent group, write $\pi:G\to G/[G,G]$ for the quotient homomorphism, and let $x\in G/[G,G]$. Then the group $\pi^{-1}(\langle x\rangle)$ has step at most $s-1$.
\end{lemma}
\begin{proof}[Remarks on the proof]
This is implicitly shown in the proofs of \cite[Propositions 4.2 \& 4.3]{nilp.frei}; it is also proved explicitly in \cite[Lemma 6.1.6 (i)]{book}.
\end{proof}

\begin{prop}[{\cite[Proposition 7.1]{nilp.frei}}]\label{prop:lower.step.in.quotient}
Let $m>0$ and $s\ge\tilde s\ge2$ be integers, and let $K\ge2$. Let $G$ be an $s$-step nilpotent group generated by a finite $K$-approximate group $A$, and let $\tilde G$ be an $\tilde s$-step nilpotent subgroup of $G$. Write $\pi:\tilde G\to\tilde G/[\tilde G,\tilde G]$ for the quotient homomorphism, and suppose that $H\subset\pi(A^m\cap\tilde G)$ is a finite group. Then there is a normal subgroup $N\lhd G$ such that $[\,\pi^{-1}(H),\ldots,\pi^{-1}(H)\,]_{\tilde s}\subset N\subset A^{K^{e^{O(s)}m}}$.
\end{prop}
\begin{proof}[Remarks on the proof]
The bounds stated here are more precise than those stated in \cite[Proposition 7.1]{nilp.frei}, but the bounds claimed here can be read out of the argument there. Alternatively, \cref{prop:lower.step.in.quotient} is proved exactly as stated here in \cite[Proposition 6.6.2]{book}.
\end{proof}

\begin{proof}[Proof of \cref{prop:key.general}]
Combine Proposition \ref{prop:pre-chang.tor.free} with Lemmas \ref{lem:slicing} and \ref{lem:x[G,G]} and \cref{prop:lower.step.in.quotient}.
\end{proof}

Before we prove \cref{prop:pre-chang.tor.free} we isolate the following lemma, which is inspired by a lemma of Tao \cite[Lemma 7.7]{tao.product.set}.
\begin{lemma}\label{lem:splitting}
Let $G$ be a group, let $N\lhd G$ be a normal subgroup, and let $\pi:G\to G/N$ be the quotient homomorphism. Let $A$ be a symmetric subset of $G$, and define a map $\varphi:\pi(A)\to A$ by choosing, for each element $x\in\pi(A)$, an element $\varphi(x)\in A$ such that $\pi(\varphi(x))=x$. Then
\begin{enumerate}[label=(\roman*)]
\item\label{eq:split.inv} for every $a\in A$ we have $a\in\left(A^2\cap N\right)\varphi(\pi(a))$; and
\item\label{eq:split.hom} for every $x,y\in G/N$ with $x,y,xy\in\pi(A)$ we have $\varphi(xy)\in\varphi(x)\varphi(y)\left(A^3\cap N\right)$.
\end{enumerate}
\end{lemma}
\begin{proof}
This is essentially just an observation: by definition of $\varphi$ we have $a\varphi(\pi(a))^{-1}\in A^2\cap N$ and $\varphi(y)^{-1}\varphi(x)^{-1}\varphi(xy)\in A^3\cap N$.
\end{proof}

\begin{lemma}\label{lem:pullback.large}
Let $G$ be a group, let $N\lhd G$ be a normal subgroup, and let $\pi:G\to G/N$ be the quotient homomorphism. Let $A$ be a finite symmetric subset of $G$, and let $P\subset\pi(A^m)$. Suppose that $|P|\ge c|\pi(A)|$. Then $|\pi^{-1}(P)\cap A^{m+2}|\ge c|A|$.
\end{lemma}
\begin{proof}
\cref{lem:fibre.pigeonhole} implies that $|N\cap A^2|\ge|A|/|\pi(A)|$, which in turn implies that $|\pi^{-1}(x)\cap A^{m+2}|\ge|A|/|\pi(A)|$ for every $x\in\pi(A^m)$. In particular, $|\pi^{-1}(P)\cap A^{m+2}|\ge|A||P|/|\pi(A)|\ge c|A|$, as desired.
\end{proof}

\begin{proof}[Proof of Proposition \ref{prop:pre-chang.tor.free}]
Write $\pi:G\to G/[G,G]$ for the quotient homomorphism, and note that $\pi(A)$ is a finite $K$-approximate subgroup of the abelian group $G/[G,G]$. Theorem \ref{thm:sanders} therefore implies that there exists a finite subgroup $H\subset G/[G,G]$, and a progression $P=\{x_1^{\ell_1}\cdots x_r^{\ell_r}:|\ell_i|\le L_i\}$ with $r\le\log^{O(1)}2K$ such that $HP\subset\pi(A^4)$ and $|HP|\ge\exp(-\log^{O(1)}2K)|\pi(A)|$. Lemma \ref{lem:pullback.large} then implies that
\begin{equation}\label{eq:pullback.large}
|\pi^{-1}(HP)\cap A^6|\ge\exp(-\log^{O(1)}2K)|A|.
\end{equation}
Now let $\varphi:\pi(A^6)\to A^6$ be an arbitrary map such that $\pi(\varphi(x))=x$ for every $x\in\pi(A^6)$. Suppose that $a\in\pi^{-1}(HP)\cap A^6$, so that there exist $h\in H$ and $\ell_1,\ldots,\ell_r\in\Z$ such that $\pi(a)=hx_1^{\ell_1}\cdots x_r^{\ell_r}$. It follows from \cref{lem:splitting} \ref{eq:split.inv} that
\begin{align*}
a&\in\left(A^{12}\cap[G,G]\right)\varphi(\pi(a))\\
&=\left(A^{12}\cap[G,G]\right)\varphi(hx_1^{\ell_1}\cdots x_r^{\ell_r}),
\end{align*}
and hence by repeated application of \cref{lem:splitting} \ref{eq:split.hom} that
\begin{align*}
a&\in\left(A^{12}\cap[G,G]\right)\varphi(h)\prod_{i=1}^r\varphi(x_i^{\ell_i})\left(A^{18}\cap[G,G]\right)\\
  &\subset\left(A^{18}\cap\pi^{-1}(H)\right)\prod_{i=1}^r\left(A^{24}\cap\pi^{-1}(\langle x_i\rangle)\right).
\end{align*}
Since $a$ was an arbitrary element of $\pi^{-1}(HP)\cap A^6$, the proposition then follows from \eqref{eq:pullback.large}.
\end{proof}

It is at this point that we diverge from the original proof of \cref{thm:old}.

\begin{prop}\label{prop:ind.tor-free.post.chang}
Let $m>0$ and $s\ge\tilde s\ge2$ be integers, and let $K,\tilde K\ge2$. Let $G$ be an $s$-step nilpotent group generated by a finite $K$-approximate group $A$, and let $\tilde A\subset A^m$ be a $\tilde K$-approximate group that generates an $\tilde s$-step nilpotent subgroup $\tilde G$ of $G$. Then there exist a normal subgroup $N\lhd G$ with $N\subset A^{K^{e^{O(s)}m}}$; an integer $r\le\log^{O(1)}2\tilde K$; finite $\tilde K^{O(1)}$-approximate groups $A_1,\ldots,A_r\subset\tilde A^{O(1)}$ such that, writing $\rho:G\to G/N$ for the quotient homomorphism, each group $\langle\rho(A_i)\rangle$ has step less than $\tilde s$; and a set $X\subset\tilde A$ of size at most $\exp(\log^{O(1)}2\tilde K)$ such that $\tilde A\subset XA_1\cdots A_r$.
\end{prop}
\begin{proof}
This is immediate from \cref{prop:key.general,lem:covering}.
\end{proof}

Using \cref{prop:ind.tor-free.post.chang} to induct on the step, we arrive at the following result.

\begin{prop}\label{prop:tor-free.post.induc}
Let $m>0$ and $s\ge\tilde s\ge1$ be integers, and let $K,\tilde K\ge2$. Let $G$ be an $s$-step nilpotent group generated by a finite $K$-approximate group $A$, and let $\tilde A\subset A^m$ be a $\tilde K$-approximate group that generates an $\tilde s$-step nilpotent subgroup $\tilde G$ of $G$. Then there exist integers $r,\ell\le e^{O(\tilde s^2)}\log^{O(\tilde s)}2\tilde K$; a normal subgroup $N\lhd G$ satisfying
\begin{equation}\label{eq:tor-free.post.induc.N}
N\subset A^{e^{O(\tilde s^2)}K^{e^{O(s)}m}\log^{O(\tilde s)}2\tilde K};
\end{equation}
finite $\tilde K^{e^{O(\tilde s)}}$-approximate groups $A_1,\ldots,A_r\subset\tilde A^{e^{O(\tilde s)}}$ such that, writing $\pi:G\to G/N$ for the quotient homomorphism, each group $\langle\pi(A_i)\rangle$ is abelian; and sets $X_1,\ldots,X_\ell\subset\tilde A^{e^{O(\tilde s)}}$ of size at most $\exp(e^{O(\tilde s)}\log^{O(1)}2\tilde K)$ such that
\[
\tilde A\subset N\prod\{A_1,\ldots,A_r,X_1,\ldots,X_\ell\},
\]
with the product taken in some order.
\end{prop}

Here, and throughout this paper, given an ordered set $X=\{x_1,\ldots,x_m\}$ of subsets and/or elements in a group $G$, we write that a product $\Pi$ of the members of $X$ is \emph{equal to $\prod X$ with the product taken in some order} to mean that there is a permutation $\xi\in\Sym(m)$ such that $\Pi=\prod_{i=1}^m x_{\xi(i)}$. If $Y=\{y_1,\ldots,y_m\}$ is another ordered set of the same number subsets and/or elements of $G$, then we say that products $\prod X$ and $\prod Y$ are \emph{taken in the same order} if $\prod X=\prod_{i=1}^m x_{\xi(i)}$ and $\prod Y=\prod_{i=1}^m y_{\xi(i)}$ for the same permutation $\xi$.

\begin{proof}
If $\tilde A$ is abelian then the proposition is trivially true with $r=1$, $\ell=0$, $A_1=\tilde A$ and $N=\{1\}$. We may therefore assume that $s\ge\tilde s\ge2$ and, by induction, that the proposition holds for all smaller values of $\tilde s$.

We start by rewriting the part \eqref{eq:tor-free.post.induc.N} of the statement we are trying to prove as
\[
N\subset A^{e^{O(\tilde s^2)}K^{e^{O(s)+O(\tilde s)}m}\log^{O(\tilde s)}2\tilde K}.
\]
This is exactly equivalent to \eqref{eq:tor-free.post.induc.N}, but writing the bound in this way makes it slightly easier to keep track of through the induction. For the same reason, at various points in the argument we use the trivial observation that any quantity bounded by $O(1)$ is also bounded by $e^{O(1)}$.

Applying Proposition \ref{prop:ind.tor-free.post.chang}, we obtain a normal subgroup $N_0\lhd G$ with $N_0\subset A^{K^{e^{O(s)}m}}$; an integer $r_0\le\log^{O(1)}2\tilde K$; finite $\tilde K^{e^{O(1)}}$-approximate groups $\tilde A_1,\ldots,\tilde A_{r_0}\subset\tilde A^{O(1)}\subset\tilde A^{e^{O(1)}}\subset A^{e^{O(1)}m}$ such that, writing $\rho:G\to G/N$ for the quotient homomorphism, each group $\langle\rho(\tilde A_i)\rangle$ has step less than $\tilde s$; and a set $X\subset\tilde A$ of size at most $\exp(\log^{O(1)}2\tilde K)$ such that
\begin{equation}\label{eq:induction.step}
\tilde A\subset X\tilde A_1\cdots\tilde A_r.
\end{equation}
Since $G/N_0$ is generated by the $K$-approximate group $\rho(A)$, we may apply the induction hypothesis to each approximate subgroup $\rho(\tilde A_i)$ of $G/N_0$ to obtain, for each $i=1,\ldots,r_0$, integers
\begin{align*}
r_i,\ell_i&\le e^{O((\tilde s-1)^2)}\log^{O(\tilde s-1)}(2\tilde K^{e^{O(1)}})\\
   & \le e^{O(\tilde s(\tilde s-1))}\log^{O(\tilde s-1)}2\tilde K;
\end{align*}
a normal subgroup $N_i\lhd G$ containing $N_0$ and satisfying 
\begin{align*}
N_i&\subset A^{e^{O((\tilde s-1)^2)}K^{e^{O(s)+O(\tilde s-1)}(e^{O(1)}m)}(e^{O(1)}\log2\tilde K)^{O(\tilde s-1)}}N_0\\
  &\subset A^{e^{O(\tilde s(\tilde s-1))}K^{e^{O(s)+O(\tilde s)}m}\log^{O(\tilde s-1)}2\tilde K}N_0;
\end{align*}
finite $\tilde K^{e^{O(\tilde s)}}$-approximate groups $A_1^{(i)},\ldots,A_{r_i}^{(i)}\subset\tilde A_i^{e^{O(\tilde s-1)}}\subset\tilde A^{e^{O(\tilde s)}}$ such that, writing $\pi_i:G\to G/N_i$ for the quotient homomorphism, each group $\langle\pi_i(A_j^{(i)})\rangle$ is abelian; and sets $X_1^{(i)},\ldots,X_{\ell_i}^{(i)}\subset\tilde A_i^{e^{O(\tilde s-1)}}\subset\tilde A^{e^{O(\tilde s)}}$ satisfying
\begin{align*}
|X_j^{(i)}|&\le\exp(e^{O(\tilde s-1)}\log^{O(1)}(2\tilde K^{e^{O(1)}}))\\
   &\le\exp(e^{O(\tilde s)}\log^{O(1)}2\tilde K)
\end{align*}
such that
\begin{equation}\label{eq:induction.hyp}
\tilde A_i\subset N_i\prod\{A_1^{(i)},\ldots,A_{r_i}^{(i)},X_1^{(i)},\ldots,X_{\ell_i}^{(i)}\},
\end{equation}
with the product taken in some order. 

Defining $N=N_1\cdots N_{r_0}$, we then have
\begin{align*}
N&\subset A^{r_0e^{O(\tilde s(\tilde s-1))}K^{e^{O(s)+O(\tilde s)}m}\log^{O(\tilde s-1)}2\tilde K}\cdot N_0\\
   &\subset A^{e^{O(\tilde s(\tilde s-1))}K^{e^{O(s)+O(\tilde s)}m}\log^{O(\tilde s)}2\tilde K}\cdot N_0\\
   &\subset A^{e^{O(\tilde s(\tilde s-1))}K^{e^{O(s)+O(\tilde s)}m}\log^{O(\tilde s)}2\tilde K}\cdot A^{K^{e^{O(s)}m}}\\
   &\subset A^{e^{O(\tilde s^2)}K^{e^{O(s)+O(\tilde s)}m}\log^{O(\tilde s)}2\tilde K}.
\end{align*}
Moreover, \eqref{eq:induction.step} and \eqref{eq:induction.hyp} imply that
\[
\tilde A\subset N\prod\{A_1^{(1)},\ldots,A_{r_1}^{(1)},\ldots,A_1^{(r_0)},\ldots,A_{r_{r_0}}^{(r_0)},X_1^{(1)},\ldots,X_{\ell_1}^{(1)},\ldots,X_1^{(r_0)},\ldots,X_{\ell_{r_0}}^{(r_0)},X\}
\]
with the product taken in some order. We also have
\begin{align*}
(r_1+\ldots+r_{r_0})\,,\,(\ell_1+\ldots+\ell_{r_0}+1)&\le r_0e^{O(\tilde s(\tilde s-1))}\log^{O(\tilde s-1)}2\tilde K+1\\
   &\le e^{O(\tilde s(\tilde s-1))}\log^{O(\tilde s)}2\tilde K+1\\
   &\le e^{O(\tilde s^2)}\log^{O(\tilde s)}2\tilde K.
\end{align*}
Finally, since every $\langle\pi_i(A_j^{(i)})\rangle$ is abelian, every $\langle\pi(A_j^{(i)})\rangle$ certainly is, so the proof is complete.
\end{proof}

\begin{prop}[{\cite[Proposition 7.3]{nilp.frei}}]\label{prop:grp.in.normal}
Let $s\in\N$ and $K\ge1$. Let $G$ be an $s$-step nilpotent group generated by a finite $K$-approximate group $A$. Let $H\subset A^m$ be a subgroup of $G$. Then there exists a normal subgroup $N$ of $G$ such that $H\subset N\subset A^{K^{e^{O(s)}m}}$.
\end{prop}
\begin{proof}[Remarks on the proof]
The bounds stated in \cite[Proposition 7.3]{nilp.frei} are less explicit than the ones claimed here; as usual, the bounds claimed here can be read out of the argument there, or alternatively found explicitly in \cite[Corollary 6.5.2]{book}.
\end{proof}

\begin{prop}\label{prop:prod.of.progs.and.small}
Let $s\in\N$ and $K\ge2$. Let $G$ be an $s$-step nilpotent group, and suppose that $A\subset G$ is a finite $K$-approximate group. Then there exist $k,\ell\le e^{O(s^2)}\log^{O(s)}2K$, ordered progressions $P_1,\ldots,P_k\subset A^{e^{O(s)}}$ of rank at most $e^{O(s)}\log^{O(1)}2K$, sets $X_1,\ldots,X_\ell\subset A^{e^{O(s)}}$ of size at most $\exp(e^{O(s)}\log^{O(1)}2K)$, and a subgroup $H<G$ normalised by $A$ satisfying $H\subset A^{K^{e^{O(s)}}}$ such that
\[
A\subset H\prod\{P_1,\ldots,P_k,X_1,\ldots,X_\ell\},
\]
with the product taken in some order.
\end{prop}
\begin{proof}
We may assume that $A$ generates $G$. Applying \cref{prop:tor-free.post.induc} with $\tilde A=A$, we obtain integers $r,t\le e^{O(s^2)}\log^{O(s)}2K$; a normal subgroup $N\lhd G$ satisfying $N\subset A^{K^{e^{O(s)}}}$; finite $K^{e^{O(s)}}$-approximate groups $A_1,\ldots,A_r\subset A^{e^{O(s)}}$ such that, writing $\pi:G\to G/N$ for the quotient homomorphism, each group $\langle\pi(A_i)\rangle$ is abelian; and sets $X_1,\ldots,X_t\subset A^{e^{O(s)}}$ of size at most $\exp(e^{O(s)}\log^{O(1)}2K)$ such that
\[
A\subset N\prod\{A_1,\ldots,A_r,X_1,\ldots,X_t\}
\]
with the product taken in some order.

For each $i=1,\ldots,r$, apply \cref{cor:sanders} to the set $\pi(A_i)$ to obtain a subgroup $H_i\subset A_i^8N\subset A^{e^{O(s)}}N$ containing $N$, an ordered progression $P_i\subset A_i^8\subset A^{e^{O(s)}}$ of rank at most $e^{O(s)}\log^{O(1)}2K$, and a set $Y_i\subset A_i\subset A^{e^{O(s)}}$ of size at most $\exp(e^{O(s)}\log^{O(1)}2K)$, such that $A_i\subset Y_iH_iP_i$. Since $G/N$ is gererated by the $K$-approximate group $\pi(A)$, applying Proposition \ref{prop:grp.in.normal} in $G/N$ implies that for each $i$ there is a normal subgroup $N_i\lhd G$ such that $H_i\subset N_i\subset A^{K^{e^{O(s)}}}N$. The subgroup $H=N_1\cdots N_r$ is then normal in $G$, and satisfies 
\begin{align*}
H&\subset A^{rK^{e^{O(s)}}}N\\
   &\subset A^{K^{e^{O(s)}}}
\end{align*}
and
\[
A\subset H\prod\{P_1,\ldots,P_r,Y_1,\ldots,Y_r,X_1,\ldots,X_t\},
\]
with the product taken in some order. This completes the proof.
\end{proof}

\begin{proof}[Proof of \cref{thm:new.gen}]
Note that if $K<2$ then $A$ is a finite subgroup of $G$, and so the theorem holds with $A=H$. We may therefore assume that $K\ge2$. Let $k,\ell\le e^{O(s^2)}\log^{O(s)}2K$,
\begin{equation}\label{eq:P.S.contained}
P_1,\ldots,P_k,X_1,\ldots,X_\ell\subset A^{e^{O(s)}},
\end{equation}
and $H\subset A^{K^{e^{O(s)}}}$ be as coming from \cref{prop:prod.of.progs.and.small}, noting in particular that
\begin{equation}\label{eq:prod.order}
AH\subset H\prod\{P_1,\ldots,P_k,X_1,\ldots,X_\ell\}.
\end{equation}
The pigeonhole principle therefore implies that there exist elements $u_1,\ldots,u_\ell$ with $u_i\in X_i$ such that the product $\prod\{P_1,\ldots,P_k,u_1,\ldots,u_\ell\}$, taken in the same order as the product in \eqref{eq:prod.order}, satisfies
\begin{align*}
\left|H\prod\{P_1,\ldots,P_k,u_1,\ldots,u_\ell\}\right|&\ge\frac{|AH|}{|X_1|\cdots|X_\ell|}\\
    &\ge\frac{|AH|}{\exp(e^{O(s)}\log^{O(1)}2K)^\ell}\\
    &\ge\frac{|AH|}{\exp(e^{O(s^2)}\log^{O(s)}2K)}.
\end{align*}
In particular, setting $Q_i=\{u_i^{-1},1,u_i\}$ for $i=1,\ldots,\ell$, we have
\[
\left|H\prod\{P_1,\ldots,P_k,Q_1,\ldots,Q_\ell\}\right|\ge\frac{|AH|}{\exp(e^{O(s^2)}\log^{O(s)}2K)},
\]
with the product again taken in the same order.

Now $\prod\{P_1,\ldots,P_k,Q_1,\ldots,Q_\ell\}$ is an ordered progression, say $P_\ord(x;L)$. The ranks of the progressions $P_i$ coming from \cref{prop:prod.of.progs.and.small} are at most $e^{O(s)}\log^{O(1)}2K$, and hence that the rank of $P_\ord(x;L)$ is at most $ke^{O(s)}\log^{O(1)}2K+\ell$, which is at most $e^{O(s^2)}\log^{O(s)}2K$. Furthermore, the containment \eqref{eq:P.S.contained} implies that
\begin{align*}
P_\ord(x;L)&\subset A^{(k+\ell)e^{O(s)}}\\
   &\subset A^{e^{O(s^2)}\log^{O(s)}2K},
\end{align*}
and \cref{prop:nilprog.equiv} therefore implies that
\[
P_\ord(x;L)\subset P_\nil(x;L)\subset\overline P(x;L)\subset A^{e^{O(s^3)}\log^{O(s^2)}2K}.
\]
This comletes the proof.
\end{proof}

\begin{remark}\label{rem:poly.bound}The polynomial bound on the product set of $A$ required to contain $H$ in \cref{thm:new.gen} comes from our applications of Propositions \ref{prop:lower.step.in.quotient} and \ref{prop:grp.in.normal}. These propositions are themselves both applications of the same result, namely \cite[Proposition 7.2]{nilp.frei}, and so the polynomial bound in \cref{thm:new.gen} can be traced to this result. It appears that a new idea would be required to improve this result in such a way as to remove the polynomial bound from \cref{thm:new.gen}.
\end{remark}

\section{Covering arguments}\label{sec:covering}
In this section we use covering arguments to prove \cref{cor:ruzsa,cor:chang.ag}. \cref{cor:ruzsa} follows from \cref{thm:new.gen} and a straightforward application of Ruzsa's covering lemma, as follows.

\begin{proof}[Proof of \cref{cor:ruzsa}]
We may assume that $A$ generates $G$. Let $H$ and $P=P_\ord(x;L)$ be as given by \cref{thm:new.gen}, noting that $H\lhd G$. Let $\pi:G\to G/H$ be the quotient homomorphism, and note that
\begin{align*}
\frac{|\pi(AP)|}{|\pi(P)|}&=\frac{|APH|}{|PH|}\\
   &\le\exp(e^{O(s^2)}\log^{O(s)}2K)\frac{|APH|}{|AH|}\\
   &=\exp(e^{O(s^2)}\log^{O(s)}2K)\frac{|\pi(AP)|}{|\pi(A)|}\\
   &\le\exp(e^{O(s^2)}\log^{O(s)}2K),
\end{align*}
the last inequality coming from the fact that $\pi(A)$ is a $K$-approximate group and $\pi(AP)\subset\pi(A)^{e^{O(s^2)}\log^{O(s)}2K}$. Applying \cref{lem:covering} in the quotient $G/H$ therefore gives a set $X\subset A$ of size at most $\exp(e^{O(s^2)}\log^{O(s)}K)$ such that $A\subset XHPP^{-1}$. Now $PP^{-1}\subset A^{e^{O(s^2)}\log^{O(s)}2K}$ is an ordered progression of rank double that of $P$, which is still at most $e^{O(s^2)}\log^{O(s)}2K$. The corollary therefore follows from \cref{prop:nilprog.equiv}.
\end{proof}

\begin{proof}[Proof of \cref{cor:chang.ag}]
We may assume that $A$ generates $G$. Let $H$ and $P_0=P_\ord(x;L)$ be as given by \cref{thm:new.gen}, noting that $H\lhd G$. Let $\pi:G\to G/H$ be the quotient homomorphism, noting that
\[
\frac{|\pi(P_0)|}{|\pi(A)|}=\frac{|P_0H|}{|AH|}\ge\exp(-e^{O(s^2)}\log^{O(s)}2K).
\]
Applying and \cref{lem:chang} in the quotient $G/H$, we therefore have
\[
t\le e^{O(s^2)}\log^{O(s)}2K
\]
and sets $S_1,\ldots,S_t\subset A$ with $|S_i|\le2K$ such that
\[
A\subset S_{t-1}^{-1}\cdots S_1^{-1}P_0^{-1}P_0S_1\cdots S_tH.
\]
Enumerating the elements of each $S_i$ as $s_{1,i},\ldots,s_{r_i,i}$ and writing
\[
Q_i=\{s_{1,i}^{\epsilon_1}\cdots s_{r_i,i}^{\epsilon_{r_i}}:\epsilon_j,\in\{-1,0,1\}\},
\]
the set $P=Q_{t-1}\cdots Q_1P_0^{-1}P_0Q_1\cdots Q_t$ is therefore an ordered progression of rank at most
\[
e^{O(s^2)}K\log^{O(s)}2K
\]
satisfying
\[
A\subset PH\subset A^{4Kt+e^{O(s^2)}\log^{O(s)}2K}H\subset A^{e^{O(s^2)}K\log^{O(s)}2K}H.
\]
The corollary therefore follows from \cref{prop:nilprog.equiv}.
\end{proof}

\section{Applications to non-nilpotent groups}\label{sec:non-nilp}
In this section we use our results to improve the bounds on the ranks of the coset nilprogressions appearing in various Freiman-type theorems in non-nilpotent groups. As in \cref{sec:details}, at various points we separate the trivial case $K<2$ from the meaningful case $K\ge2$ so as to avoid the need for multiplicative constants.

Our first corollary improves an earlier result of the author for residually nilpotent groups \cite[Corollary 1.4]{resid}, and partially improves on \cref{cor:ruzsa} for large values of $s$.
\begin{corollary}\label{cor:resid}
Let $K\ge1$. Let $G$ be a residually nilpotent group, and suppose that $A\subset G$ is a finite $K$-approximate group. Then $A$ is contained in the union of at most $\exp(K^{O(1)})$ left translates of a coset nilprogression $P\subset A^{O_K(1)}$ of rank at most $\exp(O(K^{12}))$ and step at most $K^6$.
\end{corollary}
This compares with the bound of $\exp(\exp(K^{O(1)}))$ on the rank of $P$ obtained previously by the author using \cref{thm:old}.
\begin{proof}
It follows from \cite[Theorem 1.2]{resid} that there exist subgroups $H\lhd N<G$ such that $H\subset A^{O_K(1)}$, such that $N/H$ is nilpotent of step at most $K^6$, and such that $A$ is contained in a union of at most $\exp(K^{O(1))}$ left cosets of $N$. \cref{lem:fibre.pigeonhole} then implies that $A$ is contained in a union of at most $\exp(K^{O(1))}$ left translates of $A^2\cap N$, which is a $K^3$-approximate group by \cref{lem:slicing}. The desired result therefore follows from applying \cref{cor:chang.ag} to the image of $A^2\cap N$ in $N/H$.
\end{proof}
\begin{remark*}
\cref{cor:resid} gives a better rank bound than \cref{cor:ruzsa} if the step of the ambient group is greater than $K^6$. It gives a better bound on number of translates of $P$ required to cover $A$ as soon as the step is logarithmic in $K$.
\end{remark*}

Our next corollary applies to linear groups over fields of prime order, and arises from combining \cref{cor:ruzsa} with a result of Gill, Helfgott, Pyber and Szab\'o \cite[Theorem 3]{gill-helf}.
\begin{corollary}\label{cor:ghps}
Let $n\in\N$ and $K\ge1$, and let $p$ be a prime. Suppose that $A\subset GL_n(\F_p)$ is a finite $K$-approximate group. Then there is a coset nilprogression $P\subset A^{K^{O_n(1)}}$ of rank at most $e^{O(n^2)}\log^{O(n)}2K$ and step at most $n$ such that $A$ is contained in the union of at most $\exp(K^{O_n(1)})$ left translates of $P$.
\end{corollary}
This compares with the bound of $K^{O_n(1)}$ on the rank of $P$ obtained by Gill, Helfgott, Pyber and Szab\'o using \cref{thm:old}.
\begin{proof}
If $K<2$ then $A$ is a finite subgroup and the corollary is trivial, so we may assume that $K\ge2$. It follows from \cite[Theorem 3]{gill-helf} that there exist subgroups $H\lhd N<GL_n(\F_p)$ such that $H\subset A^{K^{O_n(1)}}$, such that $N/H$ is nilpotent of step at most $n$, and such that $A$ is contained in a union of at most $\exp(K^{O_n(1))}$ left cosets of $N$. \cref{lem:fibre.pigeonhole} then implies that $A$ is contained in a union of at most $\exp(K^{O_n(1))}$ left translates of $A^2\cap N$, which is a $K^3$-approximate group by \cref{lem:slicing}. The desired result therefore follows from applying \cref{cor:ruzsa} to the image of $A^2\cap N$ in $N/H$.
\end{proof}

One can obtain a similar result in characteristic zero by combining \cref{cor:chang.ag} with a result of Breuillard, Green and Tao \cite[Theorem 2.5]{bgt.lin}, as follows.
\begin{corollary}\label{cor:bgt}
Let $n\in\N$ and $K\ge1$, and let $\Bbbk$ be a field of characterisic zero. Suppose that $A\subset GL_n(\Bbbk)$ is a finite $K$-approximate group. Then there is a coset nilprogression $P_1\subset A^{K^{O_n(1)}}$ of rank at most $e^{O(n^2)}\log^{O(n)}2K$ such that $A$ is contained in the union of at most $\exp(\log^{O_n(1)}2K)$ left translates of $P_1$, and a coset nilprogression $P_2\subset A^{K^{O_n(1)}}$ of rank at most $e^{O(n^2)}K^3\log^{O(n)}2K$ such that $A$ is contained in the union of at most $K^{O_n(1)}$ left translates of $P_2$.
\end{corollary}
\begin{proof}
If $K<2$ then $A$ is a finite subgroup and the corollary is trivial, so we may assume that $K\ge2$. It then follows from \cite[Theorem 2.5]{bgt.lin} that $A$ is contained in a union of at most $K^{O_n(1)}$ left cosets of a nilpotent subgroup $N$ of $GL_n(\Bbbk)$ of step at most $n-1$, and hence from \cref{lem:fibre.pigeonhole} that $A$ is contained in a union of at most $K^{O_n(1)}$ left translates of $A^2\cap N$. \cref{lem:slicing} implies that $A^2\cap N$ is a $K^3$-approximate group, and so the existence of $P_1$ follows from \cref{cor:ruzsa} and the existence of $P_2$ follows from \cref{cor:chang.ag}.
\end{proof}

In the special case in which $\Bbbk=\C$, an argument of Breuillard and Green shows that the coset nilprogression appearing in \cref{cor:bgt} can be replaced with simply a nilprogression, as follows.
\begin{corollary}\label{cor:bg}
Let $n\in\N$ and $K\ge1$. Suppose that $A\subset GL_n(\C)$ is a finite $K$-approximate group. Then there is a nilprogression $P_1\subset A^{e^{O(n^3)}\log^{O(n^2)}2K}$ of rank at most $e^{O(n^2)}\log^{O(n)}2K$ such that $A$ is contained in a union of at most $\exp(e^{O(n^2)}\log^{O(n)}2K)$ left translates of $P_1$, and a nilprogression $P_2\subset A^{{e^{O(n^3)}K^{3n+3}\log^{O(n^2)}2K}}$ of rank at most $e^{O(n^2)}K^3\log^{O(n)}2K$ such that $A$ is contained in a union of at most $K^{O_n(1)}$ left translates of $P_2$.
\end{corollary}

For the convenience of the reader we reproduce the Breuillard--Green argument giving \cref{cor:bg}. The argument is facilitated by the following two general results about complex linear groups, in which we write $\Upp_n(\C)$ to mean the group of upper-triangular $n\times n$ complex matrices.
\begin{theorem}[Mal'cev \cite{malcev}; see also {\cite[Theorem 3.6]{wehrfritz}}]\label{thm:malcev}
Let $n\in\N$, and suppose that $G<GL_n(\C)$ is a soluble subgroup. Then $G$ contains a normal subgroup $U$ of index at most $O_n(1)$ that is conjugate to a subgroup of $\Upp_n(\C)$.
\end{theorem}

\begin{prop}[{\cite[Proposition 3.2]{bg.sol}}]\label{prop:red.tf}
Let $n,s\in\N$, and let $N$ be an $s$-step nilpotent subgroup of $\Upp_n(\C)$. Then there is a torsion-free $s$-step nilpotent group $\Gamma$ such that $N$ embeds into $\R^n/\Z^n\times\Gamma$.
\end{prop}
\begin{proof}[Remarks on the proof]
Although \cite[Proposition 3.2]{bg.sol} does not include the statement that $\Gamma$ has the same step as $N$, one can easily obtain this by replacing $\R^n/\Z^n\times\Gamma$ with $(\R^n/\Z^n)N$.
\end{proof}

\begin{proof}[Proof of \cref{cor:bg}]
We follow part of the proof of \cite[Corollary 1.5]{bg.sol}. It follows from \cite[Theorem 2.5]{bgt.lin} that $A$ is contained in a union of at most $K^{O_n(1)}$ left cosets of a nilpotent subgroup $N$ of $GL_n(\C)$ of step at most $n-1$. By \cref{thm:malcev} we may assume that $N\subset\Upp_n(\C)$, and then by \cref{prop:red.tf} we may assume that there exists a torsion-free $(n-1)$-step nilpotent group $\Gamma$ such that $N=\R^n/\Z^n\times\Gamma$. \cref{lem:fibre.pigeonhole} then implies that $A$ is contained in a union of at most $K^{O_n(1)}$ left translates of $A^2\cap(\R^n/\Z^n\times\Gamma)$.

Set $B=A^2\cap\left(\left[-\frac{1}{8},\frac{1}{8}\right]^n\times\Gamma\right)$. The set $\left[-\frac{1}{16},\frac{1}{16}\right]^n\times\Gamma$ is a $2^n$-approximate group; since $(\left[-\frac{1}{16},\frac{1}{16}\right]^n\times\Gamma)^2=\left[-\frac{1}{8},\frac{1}{8}\right]^n\times\Gamma$ and $(\left[-\frac{1}{16},\frac{1}{16}\right]^n\times\Gamma)^8=\R^n/\Z^n\times\Gamma$, \cref{lem:slicing} therefore implies that $B$ is a $2^{3n}K^3$-approximate group, and that $A^2\cap(\R^n/\Z^n\times\Gamma)$ is contained in a union of at most $2^{7n}K$ left translates of $B$.

Let $\ph:\R^n/\Z^n\times\Gamma\to\left(-\frac{1}{2},\frac{1}{2}\right]^n\times\Gamma\subset\R^n\times\Gamma$ be the obvious lift. The restriction of $\ph$ to $\left[-\frac{1}{8},\frac{1}{8}\right]^n\times\Gamma$ is a Freiman $3$-homomorphism, so \cref{lem:fr.hom.ap.grp} implies that $\ph(B)$ is a $2^{3n}K^3$-approximate subgroup of the torsion-free $s$-step nilpotent group $\R^n\times\Gamma$. \cref{cor:ruzsa,cor:chang.ag} therefore imply the existence of a nilprogression $Q_1\subset\ph(B)^{e^{O(n^3)}\log^{O(n^2)}2K}$ of rank at most $e^{O(n^2)}\log^{O(n)}2K$ such that $\ph(B)$ is contained in a union of at most $\exp(e^{O(n^2)}\log^{O(n)}2K)$ left translates of $Q_1$, and a nilprogression $Q_2\subset\ph(B)^{{e^{O(n^3)}K^{3n+3}\log^{O(n^2)}2K}}$ of rank at most $e^{O(n^2)}K^3\log^{O(n)}2K$ such that $\ph(B)\subset Q_2$.

Write $\pi:\R^n\times\Gamma\to\R^n/\Z^n\times\Gamma$ for the quotient homomorphism, and note that $\pi(\ph(B))=B$, so that $A$ is contained in a union of at most $K^{O_n(1)}$ left translates of $\pi(\ph(B))$. The corollary therefore follows from setting $P_1=\pi(Q_1)$ and $P_2=\pi(Q_2)$.
\end{proof}

\begin{remark*}
The reason for tradeoff between the rank of $P$ and the number of translates of it required to cover $A$ in \cref{cor:bgt,cor:bg} is that the bound on the number of cosets of $N$ needed to cover $A$ in \cite[Theorem 2.5]{bgt.lin} is stronger than the bound on the number of translates of the coset nilprogression needed to cover $A$ in \cref{cor:ruzsa}. This tradeoff does not occur in \cref{cor:ghps}, as the corresponding bounds in \cite[Theorem 3]{gill-helf} are weaker.
\end{remark*}

\end{document}